\documentclass[10pt]{amsart}

\usepackage{amsmath, amsthm, amssymb}
\usepackage[shortalphabetic]{amsrefs}

\allowdisplaybreaks[2]

\newcommand{\R}{\mathbb{R}}
\newcommand{\Z}{\mathbb{Z}}
\newcommand{\N}{\mathbb{N}}

\newcommand{\SI}{\mathcal{I}}
\newcommand{\SD}{\mathcal{D}}
\newcommand{\ST}{\mathcal{T}}
\newcommand{\SN}{\mathcal{N}}
\renewcommand{\S}{\mathbb{S}}
\newcommand{\IP}[2]{\left<#1,#2\right>}
\newcommand{\vn}[1]{\lVert#1\rVert}

\newcommand{\Ko}{K_{\text{$\mspace{-1mu}o\mspace{-1mu}s\mspace{-1mu}c$}}}

\newcommand{\kav}{\overline{k}}

\newtheorem{thm}{Theorem}[section]
\newtheorem*{thm*}{Theorem}
\newtheorem*{conj}{Conjecture}

\newtheorem{lem}[thm]{Lemma}
\newtheorem{cor}[thm]{Corollary}

\theoremstyle{definition}
\newtheorem{rmk}[thm]{Remark}
\begin{document}

\title{Convergence for global curve diffusion flows}
\author{Glen Wheeler}

\address{University of Wollongong\\
Wollongong\\
NSW 2522\\
Australia}
\email{glenw@uow.edu.au}
\subjclass[2000]{53C44 \and 58J35} 

\begin{abstract}
In this note we establish exponentially fast smooth convergence for global
curve diffusion flows, and discuss open problems relating embeddedness to
global existence (Giga's conjecture) and the shape of Type I
singularities (Chou's conjecture).
\end{abstract}
\maketitle

\section{Introduction}

The curve diffusion flow is the one-parameter family of immersed curves
$\gamma:\S^1\times[0,T)\rightarrow\R^2$ with normal velocity equal to
$-\text{grad}_{H^{-1}}(L(\gamma))$, that is
\begin{equation}
\partial^\perp_t\!\gamma = -k_{ss}\nu.
\label{CD}
\tag{CD}
\end{equation}

In \cite{W13}, the author proves that if the isoperimetric ratio and
$L^2$-normalised oscillation of curvature are close to their value on any
circle (with an explicit estimate on the constant) then the flow exists
globally and converges to a circle.
However the rate of convergence was not established.

In this note we prove a general convergence principle that can be summarised as:
\begin{align*}
	T = \infty \quad \Longrightarrow \quad \gamma&\text{ converges exponentially fast to an $\omega$-circle}\\ &\qquad\text{ in the smooth topology.}
\end{align*}

\begin{thm}
\label{TMmain}
Suppose $\gamma:\S\times[0,\infty)\rightarrow\R^2$ is a global curve diffusion flow
with smooth initial data $\gamma_0$ that has winding number $\omega$.
Then $\gamma$ converges exponentially fast to a round $\omega$-circle in the
smooth topology, with explicit estimates (here $m\in\N$)
\[
	\vn{k_{s^m}}_2^2 \le c_me^{-\kav_0^4\,t}
	\,,
\]
with
\[
	\SD(t) \le \SD_0 e^{-4\kav_0^4\,t}
	\quad\text{ and }\quad
	\Ko(t) \le c_0e^{-2\kav_0^4\,t}
	\,.
\]
The constants $c_m$ depend only on the initial data $\gamma_0$.
\end{thm}

The quantities $\SD$ and $\Ko$ are the isoperimetric defect and the $L^2$-normalised oscillation of curvature:
\[
	\SD = L^2 - 4\omega\pi A
	\quad\text{ and }\quad
	\Ko = L\int_\gamma (k-\kav)^2\,ds
\]
where $L$ and $A$ are the length and signed enclosed area of $\gamma$, and
$k$, $\kav$ are the curvature scalar and the average of $k$ respectively.
Section 2 contains the proof of Theorem \ref{TMmain}.

\begin{rmk}
Note that no smallness condition is required.
\end{rmk}

\begin{rmk}
The regularity hypothesis is not optimal.
\end{rmk}

\begin{rmk}
Theorem \ref{TMmain} is stated for flows of immersed multiply-covered circles
which makes it applicable to the case considered by Miura-Okabe \cite{MO19}.
However the curve diffusion flow starting from a multiply-covered circle is not
expected to be stable in general -- perturbations that unbalance the length of
different leaves are should lead to finite-time
singualrities.\footnote{Rigorously establishing this is an interesting open
problem.}
Miura-Okabe use a rotational symmetry assumption so that only perturbations
that preserve the balance of the different leaves are studied, and the power of
this hypothesis can be seen in the isoperimetric inequality for immersed
rotationally symmetric curves that they establish.
That is, preservation of the smallness of the isoperimetric ratio (see Lemma
\ref{CYisoine}) is proved in \cite{MO19} for rotationally
symmetric curves \emph{without assuming global existence}.
This is the decisive new ingredient that allows their result on the curve
diffusion flow to go through.
\end{rmk}

Yoshikazu Giga famously conjectured around ten years ago that:

\begin{conj}[Giga's Conjecture]
Suppose $\gamma:\S\times[0,T)\rightarrow\R^2$ is a curve diffusion flow
with smooth initial data $\gamma_0$ that has the property:
\[
	\text{$\gamma(\cdot,t)$ is an embedding for each $t\in[0,T)$.}
\]
Then $T=\infty$.
\end{conj}

If this conjecture holds, our Theorem \ref{TMmain} implies: \emph{an embedded
curve diffusion flow converges smoothly and exponentially fast to a round
circle}.

Note that in the case of the curve diffusion flow in homotopy classes outside
$\omega$-circles, shrinkers are expected and one explicit example is known (the
Lemniscate of Bernoulli, see \cite{EGMWW15}).
Indeed, Chou \cite{C03} suggests that parabolic rescalings of Type I singularities (a definition also due to Chou) are self-similar solutions to the curve diffusion flow:

\begin{conj}[Chou's Conjecture]
Suppose $\gamma:\S\times[0,T)\rightarrow\R^2$ is a curve diffusion flow with $T<\infty$ that satisfies the estimate
\begin{equation}
\label{EQtypeI}
	\vn{k}_2^2(t) \le C(T - t)^{-1/4}
	\,,
\end{equation}
for some $C\in\R$, and $t\in[0,T)$.

Then a parabolic rescaling (we assume the centre of mass of $\gamma$ is the origin)
\[
	\eta(s,t) = (T-t)^{-\frac{1}{4}}\gamma(s,t)
\]
about final time yields a self similar solution $\eta$ to the curve diffusion
flow, that is, $\eta$ solves
\begin{equation}
\label{EQshrinker}
	\IP{\eta}{\nu^\eta} = 4k^\eta_{ss}
	\,.	
\end{equation}
\end{conj}
Curve diffusion flows satisfying \eqref{EQtypeI} are said to be \emph{Type I}.
As suggested by Chou \cite{C03}, this conjecture can be
approached by studying all solutions to the shrinker equation
\eqref{EQshrinker}, and classifying all such solutions is an interesting problem.
Since area is constant along the curve diffusion flow, and shrinkers must eventually enclose zero area, all shrinkers have zero signed enclosed area.
This means that all solutions to \eqref{EQshrinker} are non-embedded (and have zero signed enclosed area).

The maximal time of smooth existence for the curve diffusion flow with
$\omega=0$ is always finite.
This is because along the flow the curvature scalar always has a zero, and so
the Wirtinger inequality implies that the flow can not exist past time $T =
\frac{L^4_0}{16\pi^4}$.
This is not a sharp estimate, as equality in the estimates used by the proof
implies $k = const$, and there is no such curve that also has $\omega=0$.
A reasonable conjecture is that the maximal existence time for a curve
diffusion flow with winding number zero and initial length $L_0$ is bounded by
that of the Lemniscate of Bernoulli starting with length $L_0$.

As far as we know the connection between non-convexity and maximal existence
time was observed first by Chou \cite{C03}.
It was later used in \cite{W13} to estimate the total waiting time before a
curve diffusion flow with small $L^2$-oscillation of curvature and
isoperimetric
ratio becomes uniformly convex.
Theorem \ref{TMmain} here combined with the waiting time estimate
\cite[Proposition 1.5]{W13} yields an estimate for the total waiting time
before \emph{any} global curve diffusion flow becomes uniformly convex.

\section*{Acknowledgements}

The author thanks his colleagues for discussions on the curve diffusion flow, especially relating to its convergence properties.
In particular he would like to thank Tatsuya Miura for extensive discussion and encouraging the dissemination of this note.
Without the encouragement of his peers this note would not have been written.
This work was completed under the financial support of the Okinawa Institute of Science and Technology, hosted by James McCoy.
He is grateful for their support and hospitality.

\section{Proof of Theorem \ref{TMmain}}

We complete the proof in a sequence of lemmata.
The hypotheses of Theorem \ref{TMmain} are not restated, but assumed throughout.

\begin{lem}
\label{LMkszero}
	There exists a subsequence of times $t_j\rightarrow\infty$ such that $\vn{k_s}_2^2(t_j)\rightarrow0$.
\label{LM1}
\end{lem}
\begin{proof}
Note that
\[
	L' = -\vn{k_s}_2^2
	\,.
\]
Integration and the monotonicity of length implies the result.
\end{proof}

As far as we are aware, Lemma \ref{LM1} was first observed by Chou \cite{C03}.

Now we observe that the signed enclosed area of the curve under the flow can not be non-positive.

\begin{lem}
\label{LMareanotzero}
	The signed enclosed area $A$ of $\gamma$ is strictly positive.
\end{lem}
\begin{proof}
Our convention is that the area of an $\omega$-circle is positive (and equal to $\omega\pi r^2$, where $r$ is its radius).
If the area is not strictly positive, then the flow can exist for at most a finite amount of time.
Since area is constant under the flow, if it is non-positive at any time, it must be so for all time.
Then, since $A(t) < 0$ implies that $k(s_t,t) = 0$ for some point $s_t$, we use the Wirtinger and Poincar\'e inequalities to conclude
\[
	L'(t) \le -\frac{\pi^2}{L^2}\vn{k}_2^2
	\le -\frac{4\pi^4}{L^3}
\]
which implies that the flow can exist at most for a finite time (in fact $T \le \frac{L_0}{16\pi^4}$).
This is a contradiction.
\end{proof}

Now set
\[
\SD = L^2 - 4\omega\pi A
\,.
\]
Recall that $\SD$ is the \emph{isoperimetric defect}.
Since length is monotone decreasing and signed enclosed area is constant in time, we immediate have:

\begin{lem}
\label{LMdefdec}
The isoperimetric defect satisfies
\[
	\SD'(t) \le 0\,.
\]
\end{lem}

By combining Lemma \ref{LMkszero}, Lemma \ref{LMareanotzero} and Lemma \ref{LMdefdec} we find the following.

\begin{cor}
\label{CYisoine}
Let $\varepsilon>0$ be arbitrary.
There exists a $t_j$ such that for all $t>t_j$,
\[
	\SI(t) := \frac{L^2}{4\omega A} \le 1 + \varepsilon
	\,.
\]
\end{cor}
\begin{proof}
The estimate
\begin{equation}
\label{EQ1}
	\SD \le \frac{L^2}{4}\Ko^\frac12
	\,,
\end{equation}
(here $\Ko$ is the $L^2$-oscillation of curvature from \cite{W13}) together
with the Poincar\'e inequality implies
\begin{equation}
\label{EQ11}
	\SD 
	    \le \frac{L^{\frac72}}{8\omega\pi} \vn{k_s}_2
	\,.
\end{equation}
The estimate \eqref{EQ1} follows using an elementary argument (see the Appendix).

Note that length is uniformly bounded below by the measure-theoretic isoperimetric inequality
\[
	L^2(t) \ge 4\pi A^+
\]
where $A^+$ is the area of $\gamma$ changed so that the orientation of every loop is positive (no segments of negative area).
Clearly we always have $A^+ \ge A$ and so (using the constancy of area along the flow)
\[
	L^2(t) \ge 4\pi A(0)
\]
follows.
This is not a sharp estimate for $\omega \ne 1$ but it is enough to give a uniform lower bound on evolving area for any winding number.
Note that the shrinking Lemniscare of Bernoulli has $A(0) = 0$ and $T<\infty$.
Lemma \ref{LMareanotzero} implies that $A(0) > 0$ and so in particular length is uniformly bounded from below by a universal constant.

Therefore the estimate \eqref{EQ11} implies that $\SD(t_j) \rightarrow 0$, which implies $\SI(t_j) \rightarrow 1$.
Now take $j$ sufficiently large that $\SI(t_j) \le 1 + \varepsilon$, where $\varepsilon>0$ is as given by hypothesis.
Then Lemma \ref{LMdefdec} implies that $\SI(t) \le 1+\varepsilon$ for all $t>0$, as required.
\end{proof}

\begin{lem} There are uniform bounds on all derivatives of curvature.
\end{lem}
\begin{proof}
We recall the estimate \eqref{EQ11} and 
\begin{equation}
\label{EQ2}
	\Ko \le \frac{L^3}{4\omega^2\pi^2} \vn{k_s}_2^2
	\,,
\end{equation}
which follows from the Poincar\'e inequality.
Each of the estimates \eqref{EQ1}, \eqref{EQ2} are uniform (from the uniform boundedness of length).

These estimates imply that the isoperimetric defect and $\Ko$ converge to zero at times $t_j$ as $j\rightarrow\infty$.
Therefore for $j$ sufficiently large, the hypotheses of \cite[Proposition 3.7]{W13} are satisfied at time $t_j$.
Given the eventual smallness of the isoperimetric ratio (Corollary \ref{CYisoine}), the argument in \cite{W13} (see also \cite{MO19})
applies to yield uniform estimates on $\Ko$.
This then implies bounds on all derivatives of curvature by the evolutionary
interpolation inequalities of Dziuk-Kuwert-Sch\"atzle \cite{DKS01}.\footnote{The argument is by now standard in the literature, see for instance
\cite[Theorem 6.4]{AMWW} for a recent example.}
\end{proof}

\begin{lem}
The isoperimetric defect decays exponentially fast:

	\[
	\SD(t) \le \SD_0 e^{-4\kav_0^4\,t}
\,.
\]
\end{lem}
\begin{proof}
Since
\begin{align*}
\SD' &= -2L\int_\gamma k_s^2\,ds
\,,
\end{align*}
the estimate \eqref{EQ1} implies
\[
\SD' \le  -\SD\frac{64\omega^4\pi^4}{L^4}
     \le  -\SD\frac{64\omega^4\pi^4}{L_0^4}
\,.
\]
Therefore
\[
	(\log\SD)' \le -4\kav_0^4
	\,,
\]
where $\kav_0$ is the initial average of the curvature scalar.
Then
\[
	\SD(t) \le \SD_0 e^{-4\kav_0^4\,t}
	\,,
\]
as required.
\end{proof}

\begin{rmk}
One may also attempt the same strategy with the isoperimetric ratio in place of the isoperimetric defect.
However, the resultant sharp estimate looks like
\[
	(\SI^{-1})' \ge \kav^4
	\,.
\]
This results in linear decay of the isoperimetric ratio -- not exponential.
\end{rmk}

\begin{lem}
All derivatives of curvature decay exponentially fast, with the explicit estimate
\[
	\vn{k_{s^m}}_2^2 \le c_me^{-\kav_0^4\,t}
	\,.
\]
\end{lem}
\begin{proof}
Since the previous lemma establishes uniform bounds for all derivatives of curvature, a Fourier series argument (taught to us by Ben Andrews many years ago and used in for example 
\cite{AMWW}; we give an explanation in the appendix).
\footnote{The author recently learned that similar Fourier series arguments (and inequalities that imply the below) also appear in \cite{NN19}.}
\[
	\Ko \le c\frac{\sqrt{\SD}}{L}
\]
where $c$ depends on the estimates for $k$ and $k_s$.
Now the isoperimetric inequality (and constancy of area) implies that $L$ is uniformly bounded from below, so the above implies that $\Ko$ satisfies
\[
	\Ko \le c_0e^{-2\kav_0^4\,t}
	\,.
\]
Integration by parts and the uniform estimates imply that for any $m$ we have
\[
	\vn{k_{s^m}}_2^2 \le c_me^{-\kav_0^4\,t}
	\,,
\]
as required.
\end{proof}

The exponential decay implies convergence of the position vector $\gamma$ to an
$\omega$-circle with a standard argument (for example this was used by Huisken
\cite{H84} in his seminal work on mean curvature flow).
Briefly, this is because we then have estimates on $\gamma$ by simply
integrating the evolution equation in time.
We may convert from arclength derivatives to arbitrary ones in a manner
analogous to \cite[proof of Theorem 3.1]{DKS01}.
Note that this integration in time can not be done with only linear decay estimates.

\appendix

\section{Bounding oscillation of curvature by the isoperimetric defect}

Both estimates used in this article are contained in the following statement.

\begin{lem}
Let $\gamma:\S\rightarrow\R^2$ be a smooth immersed curve.
Then
\begin{equation}
\label{EQappendix}
	\frac{16}{L^4} \SD^2
	\le \Ko
	\le c\frac{\SD^\frac12}{L}
	\,.
\end{equation}
where $c$ is such that
\[
	L^3\vn{k_s}_2^2 + L^{5/2}\vn{k}_6^3\Ko^\frac12 \le c\,.
\]
\end{lem}

Note that when the upper estimate is used, curvature and its derivative(s) are uniformly bounded, so the hypothesis involving $c$ is satisfied.

\subsection{Lower estimate}

Let us prove first the lower estimate.
To begin, we translate the curve $\gamma$ to $\tilde\gamma$ so that
\[
	\int \tilde\gamma\,ds = 0
\]
and the estimate
\[
	|\tilde\gamma| \le \frac{L}{4}
\]
holds.
The lower estimate in \eqref{EQappendix} is invariant under translation (this
is true of both sides), so if we can prove the estimate for $\tilde\gamma$ it
will be true for $\gamma$ also.

Recall that
\[
	A = -\frac12\int \IP{\tilde\gamma}{\nu}\,ds
	\,.
\]
Then integration by parts implies
\begin{align*}
	\SD &= \int L + 2\omega\pi \IP{\tilde\gamma}{\nu}\,ds
\\
	&= \int -kL\IP{\tilde\gamma}{\nu} + 2\omega\pi \IP{\tilde\gamma}{\nu}\,ds
\\
	&= \int (-kL + 2\omega\pi) \IP{\tilde\gamma}{\nu}\,ds
\,.
\end{align*}
Since $\kav = \frac{1}{L}\int k\,ds = \frac{2\omega\pi}{L}$, we have $\kav-k = \frac1L(-kL + 2\omega\pi)$.
Therefore
\[
\SD 
	\le \int |-kL + 2\omega\pi| \frac{L}{4}\,ds
	\le \frac{L^2}4 \vn{k-\kav}_1
\,.
\]
The lower estimate in \eqref{EQappendix} follows now from H\"older's inequality.

\begin{rmk}
The estimate
\[
	\bigg| \IP{\gamma}{\nu} - \frac1L\int \IP{\gamma}{\nu}\,ds \bigg| \le L
\]
is false in general.
Take for example a circle with unit radius centred at $(P,0)$ in the plane.
Then, 
\[
	\sup \bigg| \IP{\gamma}{\nu} - \frac1L\int \IP{\gamma}{\nu}\,ds \bigg|
	\ge (P+1) + \frac{2\pi}{2\pi}
	= P+2
\,.
\]
Since $P$ is arbitrary, this quantity is unbounded.
This estimate is used in \cite{NN19}, although it is not a major issue
(translation invariance as used here can also be used there).
\end{rmk}

%
%

\subsection{Upper estimate}

\vspace{3mm}
~\\

{\bf Acknowledgement.} We would like to express our great gratitiude to Ben
Andrews, who taught us the technique we are about to use (and many, many other
things).
In particular Ben demonstrated how to prove a great variety of inequalities for
curves, and how they can be applied to many different kinds of curvature flow.
This occured while the author was his guest at the Mathematical Sciences Center
at Tsinghua University in Beijing, 2011.
These results should be thought of as due to Ben Andrews and not the author.
\\

As in \cite[Proof of Theorem 6.1]{AMWW} we work in the complex plane.
Let us identify $\gamma(s) = (x(s), y(s))$ with $s\mapsto x(s) + iy(s)$.
The idea is to use a Fourier series decomposition for $\gamma$ to reduce the
question of proving geometric inequalities to properties of infinite series of
integers.


We write
\[
	\gamma(s) = \sum_{p\in\Z} \hat\gamma(p) c_p(s)
\]
where the coefficients $c_p$ are given by
\[
	c_p(s) = \frac1{\sqrt{L}} e^{2i\omega\pi ps/L}\,,
\]
with projection $\hat\gamma$ defined via
\[
	\hat\gamma(p) = \int_\gamma \gamma(s)\overline{c_p}(s)\,ds
	\,.
\]

Curvature arises by differentiating $\gamma$, which is the same as differentiating the Fourier decomposition.
The basic relation is:

\begin{lem}
\label{LMinteq}
For $q\ge2$ we have
\[
	\sum_{p\in\Z} p^q |\hat\gamma(p)|^2
	= -\frac{i}{(2i\omega\pi/L)^q}\int kQ_{q-1}\,ds
	= \frac{i^{-q-1}L^q}{(2\omega\pi)^q} \int kQ_{q-1}\,ds
\]
where $Q_q = (\partial_s^{q-1}\gamma)\overline{\partial_s\gamma}$.
For the $q=1$ case we have
\[
	 \sum_{p\in\Z} p |\hat\gamma(p)|^2
	= \frac{iL}{2\omega\pi}\int Q_1\,ds
	\,.
\]
\end{lem}
\begin{proof}
We provide a sketch with the essential steps.
First, observe that for $q\ge2$ we have
\begin{equation}
\label{EQapp1}
	\int Q_q\,ds
	= i\int kQ_{q-1}\,ds
	\,.
\end{equation}
This is because $Q_{q} = ikQ_{q-1} + \partial_sQ_{q-1}$ (this equality can be proved by a striaghtforward induction argument).
But now we use the series decomposition with orthonormality of $c_p$ to conclude
\begin{align*}
	\int Q_q\,ds
	&= \int (\partial_s^{q-1}\gamma)\overline{\partial_s\gamma}\,ds
\\
	&= \int \bigg(
			\sum_p (2i\omega\pi/L)^{q-1}p^{q-1} c_p(s) \hat\gamma(p)
		\bigg)
	        \overline{\bigg(
			\sum_p (2i\omega\pi/L)p c_p(s) \hat\gamma(p)
		\bigg)}
	\,ds
\\
	&= -(2i\omega\pi/L)^q \sum_{p\in\Z} p^q |\hat\gamma(p)|^2\,.
\end{align*}
In the $q=1$ case this finishes the proof; for $q\ge2$ we combine this with \eqref{EQapp1} to finish the proof.
\end{proof}


Lemma \ref{LMinteq} allows us to simply express the isoperimetric defect and
the $L^2$ oscillation of curvature in terms of infinite series.
To see this, we consider the cases of $q = 1, 2, 3, 4, 5, 6$ in Lemma \ref{LMinteq}.
For each of computation we make the notation
\[
	\ST = \IP{\gamma}{\partial_s\gamma}\,,\quad\text{and}\quad \SN = \IP{\gamma}{\nu}\,.
\]
Note that $\partial_s\ST = 1 + k\SN$ and $\partial_s\SN = -k\ST$.
We calculate
\begin{align*}
(q=1)&&
	\sum_{p\in\Z} p |\hat\gamma(p)|^2
	&= \frac{iL}{2\omega\pi}\int Q_1\,ds
	 = \frac{iL}{2\omega\pi}\int (x+iy)(x_s-iy_s)\,ds
\\&&
	&= \frac{iL}{2\omega\pi}\int \ST + i\SN\,ds
	 = \frac{iL}{2\omega\pi}\int i\SN\,ds
	 = \frac{LA}{\omega\pi}
\end{align*}
\begin{align*}
(q=2)&&
	\sum_{p\in\Z} p^2 |\hat\gamma(p)|^2
	&= \frac{i^{-3}L^2}{(2\omega\pi)^2} \int kQ_1\,ds
\\&&
	&= i\frac{L^2}{(2\omega\pi)^2} \int k\ST + ik\SN\,ds
	 = \frac{L^3}{(2\omega\pi)^2}
\end{align*}
\begin{align*}
(q=3)&&
	\sum_{p\in\Z} p^3 |\hat\gamma(p)|^2
	&= \frac{i^{-4}L^3}{(2\omega\pi)^3} \int kQ_2\,ds
	 = \frac{L^3}{(2\omega\pi)^3} \int ik^2Q_1 + k\partial_sQ_1\,ds
\\&&
	&= \frac{L^3}{(2\omega\pi)^3} \int (ik^2 + k\partial_s)\ST + i(ik^2 + k\partial_s)\SN\,ds
\\&&
	&= \frac{L^3}{(2\omega\pi)^3} \int k\,ds
	= \frac{L^3}{(2\omega\pi)^2}
\end{align*}
In fact, $Q_2 = (\partial_s\gamma)(\overline{\partial_s\gamma}) = 1$.
We continue to calculate:
\begin{align*}
(q=4)&&
	\sum_{p\in\Z} p^4 |\hat\gamma(p)|^2
	&= \frac{i^{-5}L^4}{(2\omega\pi)^4} \int kQ_3\,ds
	 = -i\frac{L^4}{(2\omega\pi)^4} \int (ik^2 + k\partial_s)Q_2\,ds
\\&&
	&= -i\frac{L^4}{(2\omega\pi)^4} \int ik^2\,ds
	 = \frac{L^4}{(2\omega\pi)^4} \int k^2\,ds
\end{align*}
\begin{align*}
(q=5)&&
	\sum_{p\in\Z} p^5 |\hat\gamma(p)|^2
	&= \frac{i^{-6}L^5}{(2\omega\pi)^5} \int kQ_4\,ds
	 = -\frac{L^5}{(2\omega\pi)^5} \int (ik^2 + k\partial_s)Q_3\,ds
\\&&
	&= -\frac{L^5}{(2\omega\pi)^5} \int (ik^2 + k\partial_s)(ik)\,ds
\\&&
	&= \frac{L^5}{(2\omega\pi)^5} \int k^3\,ds
	\,.
\end{align*}
\begin{align*}
(q=6)&&
	\sum_{p\in\Z} p^6 |\hat\gamma(p)|^2
	&= \frac{i^{-7}L^6}{(2\omega\pi)^6} \int kQ_5\,ds
	 = i\frac{L^6}{(2\omega\pi)^6} \int (ik^2 + k\partial_s)(-k^2+ik_s)\,ds
\\&&
	&= i\frac{L^6}{(2\omega\pi)^6} \int (-ik^4 -k^2k_s - 2k^2k_s + ikk_{ss})\,ds
\\&&
	&= \frac{L^6}{(2\omega\pi)^6} \int (k^4 + k_s^2)\,ds
	\,.
\end{align*}
Although we won't need it here, one can continue to calculate more of these
identities for greater values of $q$.
They can be used to establish higher order inequalities -- we include an example (inequality \eqref{EQapp2}).
Here are the next two:
\begin{align*}
(q=7)&&
	\sum_{p\in\Z} p^7 |\hat\gamma(p)|^2
	&= \frac{i^{-8}L^7}{(2\omega\pi)^7} \int kQ_6\,ds
\\&&
	&= \frac{L^7}{(2\omega\pi)^7} \int (ik^2 + k\partial_s)(-ik^3-3kk_s+ik_{ss})\,ds
\\&&
	&= \frac{L^7}{(2\omega\pi)^7} \int k^5 - 6ik^3k_s - 4k^2k_{ss} - 3kk_s^2 + ikk_{sss}\,ds
\\&&
	&= \frac{L^7}{(2\omega\pi)^7} \int k^5 + 5kk_s^2\,ds
	\,.
\end{align*}
\begin{align*}
(q=8)&&
	\sum_{p\in\Z} p^8 |\hat\gamma(p)|^2
	&= \frac{i^{-9}L^8}{(2\omega\pi)^8} \int kQ_7\,ds
\\&&
	&= -i\frac{L^8}{(2\omega\pi)^8} \int (ik^2 + k\partial_s)(k^4 - 6ik^2k_s - 4kk_{ss} - 3k_s^2 + ik_{sss})\,ds
\\&&
	&= -i\frac{L^8}{(2\omega\pi)^8} \int ik^6 + 10k^4k_s - 15ik^2k_s^2 - 10ik^3k_{ss}
\\&&
	&\qquad\qquad\qquad\qquad\qquad\qquad
						- 10kk_sk_{ss} - 5k^2k_{sss}
						+ ikk_{s^4} \,ds
\\&&
	&= \frac{L^8}{(2\omega\pi)^8} \int k^6 + 15k^2k_s^2 + 10k^3k_{ss} + 10ikk_sk_{ss} + 5ik^2k_{sss} + k_{ss}^2 \,ds
\\&&
	&= \frac{L^8}{(2\omega\pi)^8} \int k^6 - 15k^2k_s^2 + k_{ss}^2 \,ds
	\,.
\end{align*}
This last identity can be used trivially to deduce the estimate
\begin{equation}
\label{EQapp2}
	15\int k^2k_s^2\,ds \le \int k^6 + k_{ss}^2\,ds
	\,.
\end{equation}
This estimate does not seem to follow from using more common methods.

Now we return to our aim of proving the upper estimate.
First, we write key quantities in terms of infinite series.
\begin{lem}The equalities
\[
	\SD = \frac{(2\omega\pi)^2}{L} \sum_{p\in\Z} p(p-1)|\hat\gamma(p)|^2
\]
and
\[
	\Ko = \frac{(2\omega\pi)^4}{L^3} \sum_{p\in\Z} p^2(p^2-1)|\hat\gamma(p)|^2
	    = \frac{(2\omega\pi)^4}{L^3} \sum_{p\in\Z} p^3(p-1)|\hat\gamma(p)|^2
\]
hold.
\label{LMdefoscapp}
\end{lem}
\begin{proof}
We calculate
\[
	\SD = L^2 - 4\omega\pi A = \frac{4\omega^2\pi^2}{L} \bigg(\frac{L^3}{(2\omega\pi)^2} - \frac{AL}{\omega\pi}\bigg)
\,.
\]
Using the expressions for $q=2$ and $q=1$, this shows
\[
	\SD = \frac{4\omega^2\pi^2}{L} \bigg(
						\sum_{p\in\Z} (p^2-p) |\hat\gamma(p)|^2
					\bigg)\,,
\]
as required.
Now
\[
	\Ko = L\int (k-\kav)^2\,ds
		= L\int k^2\,ds - (2\omega\pi)^2
\]
which, using the expressions for $q=4$ and $q=2$, gives
\[
	\Ko = \frac{(2\omega\pi)^4}{L^3} 
					\sum_{p\in\Z} p^4 |\hat\gamma(p)|^2
		- \frac{(2\omega\pi)^4}{L^3}
					\sum_{p\in\Z} p^2 |\hat\gamma(p)|^2
	= \frac{(2\omega\pi)^4}{L^3}\sum_{p\in\Z} (p^4-p^2) |\hat\gamma(p)|^2
	\,.
\]
For the second term we could have used the equality for $q=3$ instead, which gives the last equality and finishes the proof.
\end{proof}


Finally, we use an inequality for infinite series (in this case it is simply
the discrete H\"older inequality) to conclude essentially the upper estimate.

\begin{lem}
\label{LMfinalapp}
	We have
\[
	\Ko^2
	\le L\SD\Big( \vn{k_s}_2^2 + L^{-1/2}\vn{k}_6^3\Ko^\frac12 \Big)
	\,.
\]
\end{lem}
\begin{proof}
We use the aforementioned discrete form of H\"older's inequality to estimate
\begin{align*}
	\Ko &= \frac{(2\omega\pi)^4}{L^3} \sum_{p\in\Z} p^3(p-1)|\hat\gamma(p)|^2
	\\
	    &\le \frac{(2\omega\pi)^4}{L^3}
	    	\bigg(\sum_{p\in\Z} p(p-1)|\hat\gamma(p)|^2\bigg)^\frac12
	    	\bigg(\sum_{p\in\Z} p^5(p-1)|\hat\gamma(p)|^2\bigg)^\frac12
		\,.
\end{align*}
Now Lemma \ref{LMdefoscapp} implies
\begin{align*}
	\Ko 
	    &\le \frac{(2\omega\pi)^4}{L^3}
		\SD^\frac12 \frac{L^\frac12}{2\omega\pi}
	    	\bigg(\sum_{p\in\Z} p^5(p-1)|\hat\gamma(p)|^2\bigg)^\frac12
		\,.
\end{align*}
The expressions for $q=5$ and $q=6$ now imply
\begin{align*}
	\Ko^2 
	    &\le \frac{(2\omega\pi)^8}{L^6}
		\SD \frac{L}{(2\omega\pi)^2}
		\bigg(\frac{L^6}{(2\omega\pi)^6}\int k^4 + k_s^2 - \frac{2\omega\pi}{L}k^3\,ds\bigg)
\\
	    &=
		L\SD\int k_s^2 + k^3(k-\kav)\,ds
\\
	    &\le 
		L\SD\Big( \vn{k_s}_2^2 + L^{-1/2}\vn{k}_6^3\Ko^\frac12 \Big)
		\,,
\end{align*}
as required.
\end{proof}

Now we finish the proof of the upper estimate.
Since by hypothesis we have
\[
	L^3\vn{k_s}_2^2 + L^{5/2}\vn{k}_6^3\Ko^\frac12 \le c\,,
\]
Lemma \ref{LMfinalapp} implies
\[
	\Ko
	\le c\frac{\SD^\frac12}{L}
\]
as required.

\end{document}